\documentclass{amsart}

\usepackage{amssymb, hyperref}
\usepackage{mathtools}
\usepackage{tikz,ifthen}
\usepackage{enumerate}
\numberwithin{equation}{section}
\DeclarePairedDelimiter\abs{\lvert}{\rvert}%
\DeclarePairedDelimiter\norm{\|}{\|}%
\makeatletter
\let\oldabs\abs
\def\abs{\@ifstar{\oldabs}{\oldabs*}}
\let\oldnorm\norm
\def\norm{\@ifstar{\oldnorm}{\oldnorm*}}
\makeatother

\theoremstyle{plain}
\newtheorem{thm}{Theorem}[section]

\newtheorem*{remark}{Remark}

\newtheorem{lem}[thm]{Lemma}
\newtheorem{prop}[thm]{Proposition}

\newtheorem*{con}{Conditions 2.1}
\theoremstyle{definition}
\newtheorem*{defn}{Sparse Balls}

\newcommand\R{\mathbb{R}}

\newcommand{\vphi}{\varphi}

\newcommand{\les}{\lesssim}

\newcommand{\ben}{\begin{enumerate}[(i)]}
\newcommand{\een}{\end{enumerate}}

\newcommand{\minab}[2]{\min_{\substack{ {#1} \\ {#2} }}}

\AtBeginDocument{%
   \def\MR#1{}
}

\begin{document}
\title{Some remarks on Fourier restriction estimates}
\author{Jongchon Kim}
\thanks{Research supported in part by the National Science Foundation.}
\address{Department of Mathematics, University of Wisconsin-Madison, Madison, WI 53706 USA}
\email{jkim@math.wisc.edu}
\begin{abstract}
We provide $L^p\to L^q$ refinements on some Fourier restriction estimates obtained using polynomial partitioning. Let $S\subset \R^3$ be a compact $C^\infty$ surface with strictly positive second fundamental form. We derive sharp $L^p(S) \to L^q(\R^3)$ estimates for the associated Fourier extension operator for $q> 3.25$ and $q\geq 2p'$ from an estimate of Guth that was used to obtain $L^\infty(S) \to L^q(\R^3)$ bounds for $q>3.25$. We present a slightly weaker result when $S$ is the hyperbolic paraboloid in $\R^3$ based on the work of Cho and Lee. Finally, we give some refinements for the truncated paraboloid in higher dimensions.
\end{abstract}
\maketitle

\section{Introduction}
Let $S\subset \R^d$ be a compact $C^\infty$ hypersurface. The Fourier transform of a function $f\in L^1(\R^d)$ is continuous, hence the restriction operator $R_S f = \hat{f}|_S$ is well-defined on $L^1(\R^d)$. However, it is impossible to restrict $\hat{f}$ to a set of zero Lebesgue measure for $f\in L^2(\R^d)$ since $\hat{f}$ is merely in $L^2(\R^d)$ in general. In the late 1960's, Stein observed that the restriction operator $R_S$ may still be defined on $L^p(\R^d)$ for some $1<p<2$ provided that the surface $S$ is appropriately curved; see \cite{Fefferman} and \cite{Stein}. This type of results have been obtained from a priori restriction estimates
\[ \norm{ \hat{f}|_S }_{L^{q}(S,d\sigma)} \leq C \norm{f}_{L^{p}(\R^d)}, \]
where $d\sigma$ is the induced Lebesgue measure on $S$.

However, for a given hypersurface $S$, it is a difficult problem to determine optimal ranges of exponents $p$ and $q$. By duality, one may reformulate restriction estimates as extension estimates
\begin{equation}\label{eqn:ext}
\norm{E_S f}_{L^q(\R^d)} \leq C \norm{f}_{L^{p}(S)},
\end{equation}
where $E_S$ is the extension operator
\[ E_S f (x) = \int_S e^{2\pi i x\cdot \xi} f(\xi) d\sigma(\xi). \]
When $S$ is the sphere $S^{d-1}$, or more generally a compact $C^\infty$ hypersurface with nonvanishing Gaussian curvature, it is conjectured that \eqref{eqn:ext} holds if and only if $q>\frac{2d}{d-1}$ and $q\geq \frac{d+1}{d-1} p'$, where $p'=p/(p-1)$. This conjecture is related to many other conjectures, including the Bochner-Riesz and the Kakeya conjectures; see, for instance, \cite{Fefferman}, \cite{Bourgain} and \cite{Tao}. While many deep results have been obtained on the restriction conjecture, it remains open in the full $p,q$ range for $d\geq 3$. 

Recently, Guth \cite{Guth} made further progress on this problem for positively curved surfaces in $\R^3$ using polynomial partitioning, a divide and conquer technique developed by Guth-Katz \cite{GK} for the Erd\H{o}s distinct distances problem. 

\begin{thm}[Guth]\label{thm:Guth}
If $S\subset \R^3$ is a compact $C^\infty$ surface with strictly positive second fundamental form, then for all $q>3.25$,
\[ \norm{E_S f}_{L^q(\R^3)} \leq C \norm{f}_{L^\infty(S)}. \]
\end{thm}
In particular, Theorem \ref{thm:Guth} improves a previous result for $q>56/17$ due to Bourgain-Guth \cite{BG}. When $S$ is the sphere or the truncated paraboloid in $\R^3$, Theorem \ref{thm:Guth} yields $L^q(S) \to L^q(\R^3)$ estimates for the extension operator $E_S$ for all $q>3.25$; see a remark after Theorem 1 in \cite{BG} or Section 19.3 in \cite{Ma}.

We refine Theorem \ref{thm:Guth} by replacing $L^\infty(S)$ with $L^p(S)$ for $p\geq q/(q-2)$, or equivalently $q \geq 2p'$. This range of exponents $p$ is sharp.
\begin{thm}\label{thm:Guth2}
If $S\subset \R^3$ is a compact $C^\infty$ surface with strictly positive second fundamental form, then for $q>3.25$ and $q \geq 2p'$,
\begin{equation}\label{eqn:sharp}
\norm{E_S f}_{L^q(\R^3)} \leq C \norm{f}_{L^p(S)}.
\end{equation}
\end{thm}
\begin{remark}
After submitting the previous version of this article to arXiv, we learned from Marina Iliopoulou that Bassam Shayya already obtained nearly sharp result for $q>3.25$ and $q > 2p'$; see \cite[Corollary 1.1 (iii)]{BS}. That result is a consequence of more general weighted restriction estimates, which are of independent interest. From his nearly sharp result, it is not hard to obtain the sharp line $q=2p'$ by using a bilinear interpolation argument from \cite{TVV}; see Section \ref{sec:bil}. In addition, we find that our proof is similar to that in \cite{BS}, although the proof given here appears to be more straightforward. In this regard, Theorem \ref{thm:Guth2} is essentially due to Bassam Shayya and we do not claim any originality for Theorem \ref{thm:Guth2}.
\end{remark}

Sharp $L^p(S) \to L^q(\R^d)$ estimates were known in the bilinear range $q>2(d+2)/d$ by the work of Tao-Vargas-Vega \cite{TVV} and Tao \cite{TaoB}; see also \cite{Wolff}. 
More recent $L^q(S) \to L^q(\R^d)$ bounds from \cite{BG}, \cite{Guth}, and \cite{Guth2} extend this range of $q$; see \cite[Section 5.2]{LRS}. In particular, when $S\subset \R^3$ is the sphere or the truncated paraboloid, Theorem \ref{thm:Guth} yields, combined with Tao's bilinear estimate \cite{TaoB}, sharp $L^p(S) \to L^q(\R^3)$ estimates for a slightly smaller range of $q$: $q>23/7=3.28\cdots$.

The main ingredient of Theorem \ref{thm:Guth2} is an estimate of Guth \cite[Theorem 2.4]{Guth} for the ``broad" contribution to $E_S f$; see Theorem \ref{thm:broad} below. Here is an overview of the proof. When $q>2p'$, Theorem \ref{thm:Guth2} follows from a variation of the proof from \cite{Guth} that Theorem \ref{thm:broad} implies Theorem \ref{thm:Guth}. Our refinement comes from the use of a parabolic rescaling argument which involves both $L^2(S)$ and $L^\infty(S)$ norms. This modification is natural in view of Theorem \ref{thm:broad}. As a result, we obtain
\[ \norm{E_S f}_{L^{3.25}(B_R)} \leq C_{\epsilon} R^\epsilon \norm{f}_{L^2(S)}^{10/13} \norm{f}_{L^\infty(S)}^{3/13}\]
for any $\epsilon>0$ and any ball $B_R$ of radius $R$, which implies, by real interpolation, $L^p(S) \to L^q(B_R)$ estimates for $q>3.25$ and $q\geq 2p'$ with the epsilon loss $R^\epsilon$. This yields Theorem \ref{thm:Guth2} for $q>2p'$ by an epsilon removal lemma; see Theorem \ref{thm:epre} in Appendix. For the case $q=2p'$, we use a bilinear interpolation argument from \cite{TVV}.

It is worth noting that Cho and Lee \cite{CL} obtained an analogue of Theorem \ref{thm:Guth} for negatively curved quadratic surfaces; see Theorem \ref{thm:hyp}. Using their ``broad" estimate, \cite[Theorem 3.3]{CL}, we obtain
\begin{thm}\label{thm:hyp2} Let $S$ be a compact quadratic surface with strictly negative Gaussian curvature in $\R^3$. Then, for all $q>3.25$ and $q>2p'$,
\begin{equation}\label{eqn:hyp2}
\norm{E_S f}_{L^q(\R^3)} \leq C \norm{f}_{L^p(S)}.
\end{equation}
\end{thm}
Lee \cite{Lee} and Vargas \cite{Vargas} obtained \eqref{eqn:hyp2} for $q>10/3$ and $q>2p'$ using bilinear estimates; see also \cite{TV}. Unlike in the case of positively curved surfaces, the end point $q=2p'$ remains open in Theorem \ref{thm:hyp2}. This is due to the fact that bilinear estimates for negatively curved surfaces require a stronger separation condition, which results in some loss in deriving linear estimates from bilinear ones; see \cite{Lee} and \cite{Vargas}. Sharp estimates at $q=2p'$ seem to be known only in the Stein-Tomas range for $q\geq 4$; see \cite{Tomas}, \cite{St}, \cite{Greenleaf}, and \cite{Stein}. 

In higher dimensions, Bourgain-Guth \cite{BG} introduced a technique to derive linear restriction estimates from the multilinear restriction estimate of Bennett-Carbery-Tao \cite{BCT}. Assume that $S_1, S_2,\cdots, S_k$ are transverse caps on the truncated paraboloid $S=\{ (\omega,|\omega|^2)\in \R^d: |\omega|\leq 1 \}$ for some $2\leq k \leq d$ and that $f_j$ is supported on $S_j$ for each $1\leq j\leq k$. The $k$-linear restriction estimate takes the following form 
\begin{equation}\label{eqn:kl}
\norm{ \prod_{j=1}^k |E_S f_j|^{1/k} }_{L^p(B_R)} \leq C_\epsilon R^\epsilon \prod_{j=1}^k \norm{f_j}_{L^2(S)}^{1/k}.
\end{equation}
It is conjectured that \eqref{eqn:kl} holds for $p\geq 2\frac{d+k}{d+k-2}$, which is already known when $k=2$ \cite{TaoB} and $k=d$ \cite{BCT}. See also \cite{Be1} and \cite{Be2} for certain sharp estimates for a class of surfaces.

Guth \cite{Guth2} formulated a weaker variant of \eqref{eqn:kl} called $k$-broad inequality and completely settled the question of optimal range of exponents $p$ for all $2\leq k\leq d$; see Theorem \ref{thm:kbroad}. Adapting the Bourgain-Guth induction on scale argument \cite{BG}, he derived new $L^p(S) \to L^p(\R^d)$ estimates for $E_S$ from the $k$-broad inequality. We remark that a part of his proof can be modified so that one obtains $L^p(S)\to L^q(\R^d)$ estimates for some $2\leq p\leq q$ for each $k$-linearity $2 \leq k \leq \frac{d}{2}+1$. 
\begin{thm} \label{thm:high2} Let $d\geq 4$ and $S$ be the truncated paraboloid. For each integer $2\leq k \leq \frac{d}{2}+1$, the operator $E_S$ obeys the estimate 
\begin{equation}\label{eqn:high2}
\norm{E_S f}_{L^{q}(\R^d)} \leq C \norm{f}_{L^{p}(S)}
\end{equation}
for all
\[ q>q(k,d)=\frac{2(d+k)}{d+k-2} \;\; \text{ and } \;\; p \geq p(k,d)=\frac{2(d-k+1)(d+k)}{(d-k+1)(d+k)-2(k-1)}.\]
\end{thm}

When $k=2$, Theorem \ref{thm:high2} recovers sharp extension estimates in the bilinear range $q>\frac{2(d+2)}{d}$ from \cite{TVV, TaoB}. When $d$ is even and $k=\frac{d}{2}+1$, then $q(k,d) = p(k,d) = 2\cdot \frac{3d+2}{3d-2}$ and Theorem \ref{thm:high2} recovers the result in \cite{Guth2}. We note that $2\leq p(k,d) < q(k,d)$ when $3\leq k<\frac{d}{2}+1$ and Theorem \ref{thm:high2} seems to be new in this range of $k$. In particular, when $d$ is odd and $k=\frac{d+1}{2}$, then $q(k,d) = 2\cdot \frac{3d+1}{3d-3}$ and $p(k,d)<q(k,d)$, thus slightly extending the range of exponents $p$.

It is expected that a better understanding on the Kakeya conjecture may lead to some further progress on the restriction problem; see, for example, \cite{BG} and \cite{Demeter}.

\subsection*{Acknowledgments} 
I am grateful to Andreas Seeger who brought to my attention the question considered in this paper. I would like to thank Betsy Stovall for pointing out the bilinear interpolation argument used in Section \ref{sec:bil}. 

\section{Preliminaries} \label{sec:Pre}
In this section, we prepare for the proof of Theorem \ref{thm:Guth2}. We recall the estimate on broad points \cite[Theorem 2.4]{Guth} and the parabolic rescaling argument in \cite{Guth}.
\subsection{Estimate on broad points}
Let $\epsilon>0$ and $B_r^2(\omega) = \{ x\in \R^2: |x-\omega|\leq r \}$. Consider a surface $S\subset \R^3$ given as the graph of a function $h:B_1^2(0) \to \R$ satisfying the following conditions for some large $L=L(\epsilon)$, say $10^6 \epsilon^{-2}$.
\begin{con}\label{con}
\begin{enumerate}
\item $0<1/2 \leq \partial^2 h \leq 2$.
\item $0=h(0)=\partial h(0)$.
\item $h$ is $C^L$, and for $3\leq l \leq L$, $\norm{\partial^l h}_{C^0} \leq 10^{-9}$.
\end{enumerate}
\end{con}

Let $K=K(\epsilon)$ be a large number. Partition $S$ into $\sim K^2$ caps $\tau$ of diameter $\sim K^{-1}$. Then we may write $f= \sum_\tau f_\tau$, where $f_\tau=f\chi_{\tau}$. 

We now introduce the concept of broad points. 
For $\alpha\in (0,1)$, $x$ is said to be $\alpha$-broad for $E_Sf$ if 
\[ \max_\tau |E_S f_\tau (x)| \leq \alpha |E_S f(x)|. \]
Define $\operatorname{Br}_\alpha E_S f(x)$ to be $|E_S f(x)|$ if $x$ is $\alpha$-broad, and zero otherwise. With $\alpha=K^{-\epsilon}$, we see that
\begin{equation}\label{eqn:broint}
|E_S f(x) | \leq \max ( \operatorname{Br}_{K^{-\epsilon}} E_S f(x), K^\epsilon \max_\tau |E_S f_\tau(x)| ).
\end{equation}

The term $\max_\tau |E_S f_\tau(x)|$ can be controlled by an induction argument using parabolic rescaling. The main difficulty lies in the estimation of $\operatorname{Br}_{K^{-\epsilon}} E_S f$. 
\begin{thm}[Guth] \label{thm:broad} For any $\epsilon>0$, there exists $K=K(\epsilon)$ and $L=L(\epsilon)$ so that if $S$ obeys Conditions \ref{con} with $L$ derivatives, then for any radius $R$,
\[ \norm{\operatorname{Br}_{K^{-\epsilon} } E_Sf}_{L^{3.25}(B_R)} \leq C_\epsilon R^\epsilon \norm{f}_2^{12/13} \norm{f}_\infty^{1/13}. \]
In fact, we may take $K(\epsilon) = e^{\epsilon^{-10}}$.
\end{thm}
This is \cite[Theorem 2.4]{Guth}. It is sharp in the sense that given the right-hand side, the exponent $3.25$ in the inequality may not be decreased. The proof of Theorem \ref{thm:broad} involves polynomial partitioning, inductions on $R$ and $\norm{f}_{L^2}$, bilinear estimates, and geometry of tubes and algebraic surfaces.

\subsection{Parabolic rescaling}
We summarize a scaling argument from \cite[Section 2.3]{Guth} as a lemma. In what follows, we identify a function on the graph of $h:U \to \R$ with a function on $U\subset \R^2$.
\begin{lem}\label{lem:resc}
Assume that $h$ satisfies Conditions \ref{con} with $L$ derivatives. Let $0<r \ll 1$ and $S_0\subset S$ be the graph of $h$ over $B^2_r(\omega_0)$ for some $\omega_0\in B^2_1(0)$. Then there exists $h_1:B^2_1(0)\to \R$ satisfying Conditions \ref{con} with $L$ derivatives such that if $S_1$ is the graph of $h_1$, then 
\[|E_{S_0} f (x)|= |E_{S_1} g (\Phi(x))|, \] 
where
\begin{align*}
g(\eta) &= f(\omega_0 + r\eta) r^2|Jh||Jh_1|^{-1},\\
\Phi({x}) &= (rx_1+r\partial_1h(\omega_0)x_3, rx_2+r\partial_2h(\omega_0)x_3, r^2 x_3).
\end{align*}
Here, $|Jh|$ and $|Jh_1|$ are Jacobian factors bounded by $\sqrt{5}$. Moreover, $h_1$ satisfies $\partial^2_{ij} h_1(\eta) = \partial^2_{ij} h(\omega_0+r\eta)$ and $\norm{\partial ^l h_1}_{C^0} \leq r^{l-2} \norm{\partial^l h}_{C^0}$ for $3\leq l\leq L$.
\end{lem}

We shall use parabolic rescaling which involves both $L^2$ and $L^\infty$ norms. The following is a version of \cite[Lemma 2.5]{Guth}.
\begin{lem}\label{lem:para} Let $S_0$ and $S_1$ as in Lemma \ref{lem:resc}. Assume that
\begin{equation}\label{eqn:s1}
\norm{E_{S_1} g}_{L^q(B_{10rR})} \leq M \norm{g}_{L^{2}(S_1)}^{1-\theta} \norm{g}_{L^{\infty}(S_1)}^{\theta}
\end{equation} 
for some $0\leq \theta \leq 1$. Then 
\[ \norm{E_{S_0} f}_{L^q(B_R)} \leq 10 r^{1+\theta - \frac{4}{q}} M\norm{f}_{L^{2}(S_0)}^{1-\theta} \norm{f}_{L^{\infty}(S_0)}^{\theta}. \]
\end{lem}
\begin{proof}
Let $g$ be as in Lemma \ref{lem:resc}. Since $\det(\Phi) = r^4$ and $\Phi(B_R) \subset B_{10rR}$, we have
\begin{align*}
\norm{E_{S_0} f}_{L^q(B_R)} &\leq r^{-4/q} \norm{E_{S_1} g}_{L^q(B_{10rR})} 
\leq r^{-4/q}M \norm{g}_{L^{2}(S_1)}^{1-\theta} \norm{g}_{L^{\infty}(S_1)}^{\theta} \\
&\leq 10 r^{1+\theta - \frac{4}{q}}M \norm{f}_{L^{2}(S_0)}^{1-\theta} \norm{f}_{L^{\infty}(S_0)}^{\theta},
\end{align*}
where we used $\norm{g}_{L^{s}(S_1)} \leq 10 r^{2/{s'}} \norm{f}_{L^{s}(S_0)}.$
\end{proof}

\section{Proof of Theorem \ref{thm:Guth2}}\label{sec:Pr}
We shall first prove Theorem \ref{thm:Guth2} when $q>3.25$ and $q>2p'$. In Section \ref{sec:bil}, we extend the result to the scaling line $q = 2p'$ by an interpolation argument.
\subsection{An extension estimate implied by Theorem \ref{thm:broad}}
The main ingredient of the proof is the following extension estimate (cf. \cite[Theorem 2.2]{Guth}).
\begin{thm}\label{thm:improv1} For any $\epsilon>0$, there exists $L=L(\epsilon)$ so that if $S$ obeys Conditions \ref{con} with $L$ derivatives, then for any radius $R$, the extension operator $E_S$ obeys the inequality 
\begin{equation}\label{eqn:improv1}
\norm{E_S f}_{L^{3.25}(B_R)} \leq C_{S,\epsilon} R^\epsilon \norm{f}_{L^2(S)}^{10/13} \norm{f}_{L^\infty(S)}^{3/13}.
\end{equation}
\end{thm}
\begin{proof}
The proof is similar to the proof of \cite[Theorem 2.2]{Guth} for the local $L^\infty(S) \to L^{3.25}(B_R)$ estimate. We use not only $L^\infty(S)$ but also some $L^2(S)$ norm, which is suggested by Theorem \ref{thm:broad}.

We may assume that $0<\epsilon<1$ and $R\geq 1$. It will be useful to use the scale $\epsilon/2$ as well\footnote{The use of $\epsilon/2$ is not necessary for estimates weaker than \eqref{eqn:improv1} where $\norm{f}_{L^2(S)}^{10/13} \norm{f}_{L^\infty(S)}^{3/13}$ is replaced by $\norm{f}_{L^2(S)}^{1-\theta} \norm{f}_{L^\infty(S)}^{\theta}$ for any $3/13<\theta\leq 1$}. Let $K=K(\epsilon/2)=e^{(\epsilon/2)^{-10}}$ and assume that $S$ obeys Conditions \ref{con} with $L(\epsilon/2)$ derivatives, where $K(\epsilon)$ and $L(\epsilon)$ are the parameters in Theorem \ref{thm:broad}. Using \eqref{eqn:broint} with $\epsilon/2$ instead of $\epsilon$, we bound $\int_{B_R} |E_S f|^{3.25}$ by
\begin{equation}\label{eqn:brop}
\int_{B_R} \operatorname{Br}_{K^{-\epsilon/2}} E_S f^{3.25} +  \sum_{\tau} \int_{B_R} |K^{\epsilon/2} E_Sf_\tau|^{3.25}.
\end{equation}
By Theorem \ref{thm:broad}, the first term in \eqref{eqn:brop} is bounded by 
\[ \big(C_{\epsilon/2} R^{\epsilon/2} \norm{f}_2^{12/13} \norm{f}_\infty^{1/13}\big)^{3.25} \leq \big({C}_{S}^{2/13} C_{\epsilon/2} R^{\epsilon/2} \norm{f}_2^{10/13} \norm{f}_\infty^{3/13} \big)^{3.25} \]
since $\norm{f}_{L^2(S)} \leq C_S \norm{f}_{L^\infty(S)}$ for some $C_S>0$. 

To handle the second term in \eqref{eqn:brop}, we use an induction on $R$. Since Theorem \ref{thm:improv1} is trivial for $R=O(1)$, we shall assume that it holds for all radii less than $R/2$ with some constant $C_{S,\epsilon} \geq 2C_S^{2/13} C_{\epsilon/2}$, and then deduce that \eqref{eqn:improv1} holds for the radius $R$ with the same constant. Since the first term in \eqref{eqn:brop} is bounded by $[\frac{1}{2} C_{S,\epsilon} R^\epsilon \norm{f}_{L^2(S)}^{10/13} \norm{f}_{L^\infty(S)}^{3/13}]^{3.25}$, the induction closes if the second term in \eqref{eqn:brop} is bounded by the same expression.

Recall that $\tau$ is a cap of diameter $\sim K$. Therefore, we may assume that $\tau$ is contained in the graph $S_0$ of $h$ over some ball $B^2_{r}(\omega_0)$ of radius $r=K^{-1}$. Let $S_1$ be the surface as in Lemma \ref{lem:resc}. As $10rR<R/2$, the induction hypothesis implies
\begin{equation*}
\norm{E_{S_1} g}_{L^{3.25}(B_{10rR})} \leq C_{S,\epsilon} (10rR)^\epsilon \norm{g}_{L^{2}(S_1)}^{10/13} \norm{g}_{L^{\infty}(S_1)}^{3/13},
\end{equation*} 
which yields
\begin{equation}\label{eqn:single}
\int_{B_R} |E_S f_\tau|^{3.25} \leq  \big( 10^{1+\epsilon} K^{-\epsilon} C_{S,\epsilon} R^\epsilon \norm{f_\tau}_{L^{2}(S)}^{10/13} \norm{f_\tau}_{L^{\infty}(S)}^{3/13} \big)^{3.25}
\end{equation}
by Lemma \ref{lem:para} and the fact that $E_S f_\tau = E_{S_0} f_\tau$.

We bound $\norm{f_\tau}_{L^{\infty}(S)}$ by $\norm{f}_{L^{\infty}(S)}$ and then sum \eqref{eqn:single} over $\tau$ using the embedding $l^2 \hookrightarrow l^{2.5}$. Then we get 
\[ \sum_{\tau} \int_{B_R} |K^{\epsilon/2} E_S f_\tau|^{3.25} \leq 
\big(  10^{1+\epsilon} K^{-\epsilon/2} C_{S,\epsilon} R^\epsilon \norm{f}_{L^{2}(S)}^{10/13} \norm{f}_{L^{\infty}(S)}^{3/13} \big)^{3.25}. \]
Therefore, the induction closes since $10^{1+\epsilon} K^{-\epsilon/2} = 10^{1+\epsilon}  e^{-(\epsilon/2)^{-9}} \leq 1/2$. 
\end{proof}

\subsection{Theorem \ref{thm:Guth2} when $q>2p'$} \label{sec:para}
From Theorem \ref{thm:improv1}, we deduce the following result by a standard argument.
\begin{thm}\label{thm:imp} If $S\subset \R^3$ is a compact $C^\infty$ surface with strictly positive second fundamental form, then for all $\epsilon>0$ and any radius $R$, the extension operator $E_S$ obeys the inequality 
\[
\norm{E_S f}_{L^{3.25}(B_R)} \leq C_{S,\epsilon} R^\epsilon \norm{f}_{L^2(S)}^{10/13} \norm{f}_{L^\infty(S)}^{3/13}.
\]
\end{thm}

For the convenience of the reader, we sketch the standard argument here following \cite[Section2.3]{Guth}. In the paper \cite{Guth}, that argument was used for the global $L^\infty(S) \to L^{q}(\R^d)$ estimates, but it would also work for our situation. By a finite decomposition of $S$ and choosing an appropriate coordinate, we may assume that $S$ is contained in the graph of a smooth function $h:B^2_1(0) \to \R$ satisfying $h(0)=\partial h(0) = 0$. By the assumption, $\partial^2 h$ is positive definite and satisfies $\Lambda^{-1} \leq \partial^2 h \leq \Lambda$ for some $\Lambda=\Lambda_S>1$. Given $L=L(\epsilon)$, we decompose $S$ into caps of diameter $r=r(\Lambda, \norm{h}_{C^L})$. Since the number of the caps depends only on $S$ and $\epsilon$, it suffices to prove the extension estimate associated with a fixed cap. We may choose $r$ sufficiently small, so that, after parabolic rescaling, $\norm{\partial^l h_1}_{C^0} \leq 10^{-10} \Lambda^{-l}$ for all $3\leq l\leq L$. Then we do a change of variable so that $\partial^2 h_1(0)$ is the identity matrix. This may increase the size of the support of $h_1$, but by a further parabolic rescaling, $h_1$ can be made to satisfy Conditions \ref{con} with $L$ derivatives. The ball $B_R$ may be dilated during these change of variables, but is contained in a ball of radius $C R$ for some constant $C=C_S$. By applying Theorem \ref{thm:improv1}, we obtain Theorem \ref{thm:imp}. 

We are now ready to deduce Theorem \ref{thm:Guth2} for $q>2p'$. First, Theorem \ref{thm:imp} immediately yields the restricted strong type $(p_0,q_0)=(13/5, 13/4)$ estimate
\[
\norm{E_S \chi_E}_{L^{3.25}(B_R)} \leq C_{S,\epsilon} R^\epsilon \norm{\chi_E}_{L^{13/5}(S)}
\]
for any measurable set $E\subset S$. Observe that $q_0= 2p_0'$. By real interpolation with the trivial $L^1 \to L^{\infty}$ estimate, we obtain strong type estimates
\begin{equation*}
\norm{E_S f}_{L^{q}(B_R)} \leq C_\epsilon R^\epsilon \norm{f}_{L^{p}(S)}
\end{equation*}
whenever $q>3.25$ and $q\geq  2p'$. Finally, we apply the epsilon removal lemma, Theorem \ref{thm:epre}, which gives Theorem \ref{thm:Guth2} for $q>2p'$.

\subsection{Bilinear argument for the case $q=2p'$}\label{sec:bil}
Following \cite{TVV}, but restricting only to smooth phases, we say that a function $h:B^2_1(0)\to \R$ is elliptic if $h$ is smooth, $h(0)=\partial h(0)=0$, and the eigenvalues of $\partial^2 h(x)$ lie in $[1-\epsilon_0,1+\epsilon_0]$ for some $0<\epsilon_0\ll 1$ for all $x\in B^2_1(0)$. We say that a surface $S$ is elliptic if $S$ is contained in the graph of an elliptic defining function $h$. 

For the proof of Theorem \ref{thm:Guth2}, it is enough to work with elliptic surfaces by the parabolic rescaling argument in Section \ref{sec:para}. Therefore, our goal is to prove that if $S$ is an elliptic surface, then
\begin{equation}\label{eqn:linear}
\norm{E_S f}_{L^q(\R^3)} \leq C \norm{f}_{L^p(S)}
\end{equation}
for $q>3.25$ and $q = 2p'$. For this, we employ a bilinear interpolation argument as in the proof of \cite[Theorem 4.1]{TVV}; see also \cite[Section 5.2]{LRS}.

Assume that $f_1$ and $f_2$ are supported in $O(1)$-separated caps $S_1$ and $S_2$, respectively, contained in an elliptic surface $S$. 
Note that \eqref{eqn:linear} implies bilinear estimates by Cauchy-Schwarz; 
\begin{equation}\label{eqn:bil}
\norm{E_S f_1 E_S f_2}_{L^{q/2}(\R^3)} \leq C \norm{f_1}_{L^{p}(S)}\norm{f_2}_{L^{p}(S)}.
\end{equation}
We say that $(1/p,1/q)$ is a bilinear pair if \eqref{eqn:bil} holds for all elliptic surfaces. Let 
\[Q = \{ (1/p,1/q)\in [0,1]^2 : q>3.25, q>2p' \}.\]
In Section \ref{sec:para}, we verified \eqref{eqn:linear} for $(1/p,1/q)\in Q$ for any compact $C^\infty$ surface $S$ with strictly positive second fundamental form. Therefore, we know that each $(1/p,1/q)\in Q$ is a bilinear pair.

Fix $q>3.25$ and $q=2p'$. In order to prove the linear estimate \eqref{eqn:linear} with this pair of exponents, it is enough to verify that there exists $\delta>0$ such that $(1/\tilde{p},1/\tilde{q})$ is a bilinear pair whenever $(1/\tilde{p},1/\tilde{q}) \in B^2_\delta(1/p,1/q)$; see \cite[Theorem 2.2]{TVV}. 

Note that the transversality of the caps $S_1$ and $S_2$ allows bilinear pairs $(1/p,1/q)$ even for some $q<2p'$. In particular, we may take a bilinear pair $(1/p_0,1/q_0)=(7/12,1/4)$ from \cite[Theorem 2.3]{TVV}; see also \cite{MVV}. We can choose a sufficiently small $\delta=\delta_{q}$ so that, for each $(1/\tilde{p},1/\tilde{q}) \in B^2_\delta(1/p,1/q)$, the line through $(1/\tilde{p},1/\tilde{q})$ and $(1/p_0,1/q_0)$ intersects $Q$ . In other words, there is a bilinear pair $(1/p_1,1/q_1)\in Q$ such that 
\[ (1/\tilde{p},1/\tilde{q}) = (1-\theta)(1/p_0,1/q_0) + \theta(1/p_1,1/q_1)\] for some $\theta\in (0,1]$. Thus, bilinear interpolation (see e.g. \cite{BL}) implies that $(1/\tilde{p},1/\tilde{q})$ is a bilinear pair, which completes the proof of Theorem \ref{thm:Guth2}.

\section{Sketch of the proof of Theorem \ref{thm:hyp2}}\label{sec:hyp}
Cho and Lee \cite{CL} obtained the following result based on the polynomial partitioning techniques from \cite{Guth}.
\begin{thm}[Cho and Lee]\label{thm:hyp} Let $S$ be a compact quadratic surface with strictly negative Gaussian curvature in $\R^3$. Then, for all $q>3.25$ and $p=q$,
\begin{equation}\label{eqn:hyp}
\norm{E_S f}_{L^q(\R^3)} \leq C \norm{f}_{L^p(S)}.
\end{equation}
\end{thm}

Theorem \ref{thm:hyp2} slightly improves the range of $p$ of Theorem \ref{thm:hyp}. It requires a few modifications of the proof of Theorem \ref{thm:hyp} analogous to those made in Section \ref{sec:Pr}. In fact, a further minor modification is necessary since the definition of broad points in \cite{CL} is slightly different due to the need of a stronger separation condition for bilinear estimates. In particular, when doing an induction on $R$, one needs to perform an additional scaling associated with thin strips of dimensions $1\times K^{-1}$. Nevertheless, arguing as in Section \ref{sec:Pr}, it can be shown that \cite[Theorem 3.3]{CL}, an estimate on broad points, yields 
\begin{thm}\label{thm:imph} Let $S\subset \R^3$ be the graph of $h(\omega_1,\omega_2) = \omega_1 \omega_2$ over the unit cube centered at the origin and $3/13<\theta\leq 1$. Then for all $\epsilon>0$ and radius $R$, the extension operator $E_S$ obeys the inequality 
\[
\norm{E_S f}_{L^{3.25}(B_R)} \leq C_{\theta,\epsilon} R^\epsilon \norm{f}_{L^2(S)}^{1-\theta} \norm{f}_{L^\infty(S)}^{\theta}.
\]
\end{thm}

This is an analogue of Theorem \ref{thm:improv1}. Note that the limiting case $\theta = 3/13$ is excluded. This is due to the additional scaling which does not shrink a ball $B_R$ to a ball of much smaller radius. However, Theorem \ref{thm:imph} is strong enough to imply Theorem \ref{thm:hyp2}.

Currently, we do not know how to extend the result to the scaling line $q=2p'$. The situation is somewhat different from the case of elliptic surfaces. In particular, when $S$ is the hyperbolic paraboloid, the bilinear estimate \eqref{eqn:bil} fails to hold for any $q<2p'$ without a stronger separation condition on $f_1$ and $f_2$; see \cite{Lee} and \cite{Vargas}. This is related to the fact that the hyperbolic paraboloid contains line segments.

\section{Some refinements in higher dimensions}\label{sec:high}
Let $d\geq 2$. Following \cite{Guth2}, we consider $L^q(B^{d-1}) \to L^p(\R^d)$ extension estimates (note the change of the role of $p$ and $q$) for the operator \[Ef(x)=\int_{B^{d-1}} e^{i(x_1 \omega_1 + \cdots + x_{d-1}\omega_{d-1} + x_d|\omega|^2)} f(\omega) d\omega,\]
where $B^{d-1}$ is the unit ball in $\R^{d-1}$. The study of the operator $E_S$ for the truncated paraboloid in $\R^d$ reduces to the study of the operator $E$, and vice versa.

Here is the basic setup for the $k$-broad inequality in \cite{Guth2} (see also \cite{BG}). Consider a covering of the unit ball $B^{d-1}$ by a collection of finitely many overlapping balls $\tau$ of radius $K^{-1}$ for some $1\ll K\ll R$. Then decompose $f$ as $f=\sum_\tau f_\tau$ where $f_\tau$ is supported on $\tau$. Let $n(\omega)\in S^{d-1}$ be a normal vector for the paraboloid in $\R^d$ at the point $(\omega,|\omega|^2)$. For a given subspace $V\subset \R^d$, we write $\tau \notin V$ if $\angle (n(\omega),v) > K^{-1}$ for all $\omega\in \tau$ and non-zero vectors $v\in V$. Otherwise, we write $\tau\in V$.

Next, consider a covering of $B_R$ by a collection of finitely many overlapping balls $B_{K^2}$, and then study $\int_{B_{K^2}} |\sum_\tau Ef_\tau|^p$ for each fixed $B_{K^2}$. Let $V\subset \R^d$ be a $(k-1)$-dimensional subspace. Then one may consider the ``broad" part $\int_{B_{K^2}} |\sum_{\tau\notin V} Ef_\tau|^p$ and the ``narrow" part $\int_{B_{K^2}} |\sum_{\tau\in V} Ef_\tau|^p$ separately. More precisely, Guth \cite{Guth} defined the $k$-broad part of $\norm{Ef}_{L^p(B_R)}^p$ by 
\[ \norm{Ef}_{\operatorname{BL}^p_{k,A}(B_R)}^p := \sum_{B_{K^2}\subset B_R} \minab{V_1,V_2, \cdots, V_A}{(k-1)-\text{subspaces of } \R^d} \max_{\tau\notin V_a \text{ for all } a} \int_{B_{K^2}} |Ef_\tau|^p \]
for a parameter $A$ and proved the following using polynomial partitioning.
\begin{thm}[Guth]\label{thm:kbroad} For any $2\leq k \leq d$, and any $\epsilon>0$, there is a constant $A$ so that the following holds (for any value of $K$):
\begin{equation}\label{eqn:kbr}
\norm{Ef}_{\operatorname{BL}^p_{k,A}(B_R)} \les_{K,\epsilon} R^\epsilon \norm{f}_{L^q(B^{d-1})},
\end{equation}
for $p\geq \frac{2(d+k)}{d+k-2}$ and $q\geq 2$.
\end{thm}

We state a version of \cite[Proposition 9.1]{Guth2} that derives extension estimates from the $k$-broad inequalities. We consider the regime $q\leq p$ which seems to be more natural in view of the restriction conjecture. 
\begin{prop}\label{prop:guth2} Suppose that for all $K,\epsilon,$ the $k$-broad inequality \eqref{eqn:kbr} holds for some $2\leq q\leq p\leq 2\cdot \frac{k-1}{k-2}$. If $p$ is in the range 
\begin{equation}\label{eqn:range}
p \geq \frac{d+1}{\frac{2d-k}{2}-\frac{d-k+1}{q}},
\end{equation}
then $E$ obeys 
\[ \norm{Ef}_{L^p(B_R)} \les_\epsilon R^\epsilon \norm{f}_{L^q}. \]
\end{prop}
Proposition \ref{prop:guth2} follows from a minor modification of the proof of \cite[Proposition 9.1]{Guth2} for the regime $q \geq p$. Therefore, we shall focus only on the part that we need to modify. Let us first sketch the proof of \cite[Proposition 9.1]{Guth2}. The $k$-broad inequality allows one to reduce the problem to the estimation of the ``$k$-narrow" part of $\norm{Ef}_{L^p(B_R)}^p$, where only $O(K^{k-2})$ many balls $\tau$ contribute to the sum $\sum_\tau E f_\tau$. After applying the $l^2$-decoupling inequality due to Bourgain \cite{BoD} to this narrow contribution (see also \cite{BD}), H\"{o}lder's inequality is used to replace the $l^2$-norm by the $l^p$-norm in order to facilitate the summation of 
\[ \left( \sum_{\tau \in V_a } \left( \int W_{B_{K^2}} |Ef_\tau|^p \right) ^{2/p} \right)^{p/2}\]
over those balls $B_{K^2} \subset B_R$. Here, $W_{B_{K^2}}$ is a weight which is roughly the characteristic function of the ball $B_{K^2}$.

Our modification for the proof of Proposition \ref{prop:guth2} lies on the ``$k$-narrow" part. After using $l^2$-decoupling, we replace the $l^2$-norm by the $l^q$-norm, which is suggested by the $L^q \to L^p$ statement. This replaces \cite[Equation (9.7)]{Guth2} with
\[ \int_{B_{K^2}} |\sum_{\tau \in {V_a}} E f_\tau |^p \leq C_\delta K^\delta K^{(k-2)(\frac{1}{2}-\frac{1}{q})p} \left( \sum_{\tau } \left( \int W_{B_{K^2}} |Ef_\tau|^p \right) ^{q/p} \right)^{p/q}\]
for some $0<\delta<\epsilon$. After the summation over $1\leq a\leq A$, we sum the above expression over those balls $B_{K^2}\subset B_R$ using Minkowski's inequality. The remainder of the proof involves the induction on scale argument using parabolic rescaling. The induction closes when \eqref{eqn:range} is satisfied.

Let us put the condition \eqref{eqn:range} in context. When $q=2$, the condition becomes the familiar Stein-Tomas range $p \geq \frac{2(d+1)}{d-1}$. When $k=2$, the condition \eqref{eqn:range} is equivalent to the necessary condition $p \geq \frac{d+1}{d-1}q'$ for the $L^q(B^{d-1}) \to L^p(\R^d)$ estimate for the extension operator $E$. When $q=p$, the condition \eqref{eqn:range} is identical to that in \cite{Guth2} for the regime $p\leq q \leq \infty$.

For $d\geq 2$ and each integer $2 \leq k \leq \frac{d}{2}+1$, define
\[ \bar{q}(k,d) := \frac{2(d-k+1)(d+k)}{(d-k+1)(d+k)-2(k-1)} .\]
This $\bar{q}(k,d)$ is found by setting the right hand side of \eqref{eqn:range} equal to $\bar{p}(k,d) := \frac{2(d+k)}{d+k-2}$ and then solving the equation for $q$. Note that $2\leq \bar{q}(k,d) < \bar{p}(k,d)$ if $2\leq k< \frac{d}{2}+1$ and $\bar{q}(k,d)=\bar{p}(k,d)$ if $k= \frac{d}{2}+1$. Therefore, the $k$-broad inequality \cite[Theorem 1.5]{Guth} and Proposition \ref{prop:guth2} yield local extension estimates
\[ \norm{Ef}_{L^{\bar{p}(k,d)}(B_R)} \leq C_\epsilon R^\epsilon \norm{f}_{L^{\bar{q}(k,d)}(B^{n-1})}, \]
which implies Theorem \ref{thm:high2} by the epsilon removal lemma, Theorem \ref{thm:epre}.

\section*{Appendix: Epsilon removal for Fourier restriction estimates}\label{sec:prep}
Let $S$ be a compact $C^\infty$ hypersurface in $\R^d$. We shall assume that $S$ is curved in the sense that the surface measure $d\sigma$ on $S$ satisfies the Fourier decay condition
\begin{equation}\label{eqn:decay}
|\widehat{d\sigma} (\xi)| \leq C (1+|\xi|)^{-\rho}
\end{equation}
for some $\rho>0$. For surfaces with non-vanishing Gaussian curvature, it is well-known that \eqref{eqn:decay} holds with the maximum decay rate $\rho= (d-1)/2$.

Tao's epsilon removal lemma \cite[Theorem 1.2]{Tao} allows one to obtain global restriction estimates from local restriction estimates of the form
\begin{equation}\label{eqn:local}
 \norm{\hat{f}|_S}_{L^q(S,d\sigma)} \leq C_{\alpha} R^\alpha \norm{f}_{L^p(B_R)} 
\end{equation}
at the expense of decreasing the exponent $p$; see also \cite{Bourgain2} and \cite{TV}. Note that \eqref{eqn:local} is the dual of the local extension estimate 
\[ 
 \norm{E_S f}_{L^{p'}(B_R)} \leq C_{\alpha} R^\alpha \norm{f}_{L^{q'}(S)}.
\]
Tao's result was stated in the case $p=q$ in \cite{Tao}, but the argument works for $p<q$ as well. We record this observation as a theorem.

\begin{thm} \label{thm:epre}
Let $1 \leq p \leq q \leq 2$ and $0<\alpha\ll 1$. Assume that we have the local restriction estimate \eqref{eqn:local} for any ball $B_R$ of radius $R$ and any smooth function $f$ supported in $B_R$. Then there is a constant $C_{d,\rho}>0$ such that we have
\begin{equation}\label{eqn:global}
 \norm{\hat{f}|_S}_{L^q(S,d\sigma)} \leq C \norm{f}_{L^s(\R^d)}  \;\;\; \text{ for } \;\;\; \frac{1}{s} > \frac{1}{p} + \frac{C_{d,\rho}}{-\ln \alpha}.
\end{equation}
\end{thm}
In fact, we may take $C_{d,\rho} = 5\ln \big((d-1)/\rho\big)$ in Theorem \ref{thm:epre}, but this is by no means optimal. Theorem \ref{thm:epre} says, in particular, that if the local estimate \eqref{eqn:local} holds for any $\alpha>0$, then the global estimate \eqref{eqn:global} holds for all $1\leq s<p$. 

It seems worth pointing out that Bourgain-Guth \cite{BG} obtained and utilized an epsilon removal result for the case $1=q< p<2$. Their result involves an additional ingredient: the Maurey-Nikishin factorization theorem. 

A main step toward Theorem \ref{thm:epre} is an extension of the local estimate \eqref{eqn:local} to a local estimate for a union of sparse balls. 
\begin{defn} Let $C(d,\rho) = (d-1)/\rho$, where $\rho$ is as in \eqref{eqn:decay}. We say that a collection of balls $\{ B_R(x_i) \}_{i=1}^N$ in $\R^d$ is sparse if $|x_i-x_j|\geq (NR)^{C(d,\rho)}$ for $i\neq j$.
\end{defn}

Given the extension of \eqref{eqn:local} for sparse balls, Theorem \ref{thm:epre} can be obtained exactly as in \cite{Tao} or \cite{BG}. Therefore, we shall be content with proving the following lemma, which is basically \cite[Lemma 3.2]{Tao}. In what follows, we write $A\les B$ if $A\leq C B$ for some constant $C>0$, which may vary from line to line.
\begin{lem}\label{lem:localsparse} Assume that the local estimate \eqref{eqn:local} holds for some $1\leq p\leq q\leq  2$. Then we have
\[
\norm{\hat{f}|_S}_{L^q(d\sigma)} \les R^\alpha \norm{f}_{L^p(\R^d)} 
\]
whenever $f$ is supported in the union of a sparse collection of balls $\{ B_R(x_i) \}_{i=1}^N$.
\end{lem}
\begin{proof}
As in \cite{Tao}, we use of the fact that \eqref{eqn:local} implies
\begin{equation}\label{eqn:nbd}
\norm{\hat{f}|_{N_{R^{-1}} S}}_{L^q(\R^d)} \les R^{-1/q} R^\alpha \norm{f}_{L^p(\R^d)},
\end{equation}
where $N_{R^{-1}} S$ is the $R^{-1}$ neighborhood of $S$.

By the support assumption, we may write $f$ as $f(x) = \sum_{i=1}^N f_i(x-x_i)$ for some $f_i$ supported in $B_R(0)$. Let $\vphi$ be a smooth function such that $|\vphi|$ is comparable to $1$ on $B_1(0)$ and $\hat{\vphi}$ is supported in $B_1(0)$.
Let $\vphi_R = \vphi(\cdot/R)$ and $\tilde{f}_i = f_i/\vphi_R$. We can write $f(x)$ as $\sum_i (\tilde{f_i} \vphi_R)(x-x_i)$. Let $e(t) = e^{-2\pi i t}$.
We claim that
\begin{equation} \label{eq:Tao}
\norm{ \sum_i e(x_i \cdot \xi) g_i *\widehat{ \vphi_R } |_S}_{L^q_\xi(d\sigma)} \les R^{1/q} \left(\sum_{i} \norm{ g_i }_{L^q(\R^d)}^q\right)^{1/q}
\end{equation}
for all $g_i \in L^q(\R^d)$ and $1 \leq q\leq 2$. 

Assume \eqref{eq:Tao} for the moment. Note that for $\xi \in S$,
\[  \hat{f}(\xi)  =  \sum_i e(x_i \cdot \xi) \widehat{\tilde{f_i}} *\widehat{ \vphi_R } (\xi)=\sum_i e(x_i \cdot \xi) [\widehat{\tilde{f_i}}|_{N_{R^{-1}} S}] *\widehat{ \vphi_R }(\xi) \] since $\widehat{ \vphi_R }$ is supported in $B_{R^{-1}}(0)$. Therefore, the proof is completed by applying \eqref{eq:Tao} with $g_i = \widehat{\tilde{f_i}} |_{N_{R^{-1}} S}$ followed by \eqref{eqn:nbd} and the embedding $l^p \hookrightarrow l^q$;
\[\left(\sum_i \norm{\tilde{f_i}}_{L^p(\R^d)}^q \right)^{1/q}  \les \left(\sum_i \norm{f_i}_{L^p(\R^d)}^q \right)^{1/q} \leq \left(\sum_i \norm{f_i}_{L^p(\R^d)}^p \right)^{1/p} = \norm{f}_{L^p(\R^d)}. \]

The estimate \eqref{eq:Tao} can be found in \cite{Tao} in a slightly different form. We give a proof for the convenience of the reader, incorporating a simplified $L^2$ estimate from \cite{BG}. It is enough to establish \eqref{eq:Tao} for $q=1$ and $q=2$ by interpolation. Consider the case $q=1$. Note that $|\widehat{\vphi_R}|\les R^d\chi_{B_{R^{-1}}(0)}$. This gives that for any $y\in \R^d$
\[ \int_S |\widehat{ \vphi_R }(\xi - y)| d\sigma(\xi) \les R^d | S \cap B_{R^{-1}}(y)| \les R. \] This finishes the proof for $q=1$ by the triangle inequality and Fubini's theorem.

When $q=2$, we shall prove \eqref{eq:Tao} with $\vphi$ replaced by $\eta$, where $\eta$ is a smooth function supported in $B_2(0)$. Then the original statement follows by writing $\vphi$ as a sum of compactly supported functions and using the rapid decay of $\vphi$ away from $B_2(0)$. Following \cite{BG}, we write $\norm{ \sum_i e(x_i \cdot \xi) g_i *\widehat{ \eta_R } |_S}_{L^2(d\sigma)}^2$ as
\begin{equation}\label{eqn:L2}
\sum_i \norm{\widehat{G_i} |_S}_{L^2(d\sigma)}^2+ \sum_{i\neq j} \int_S e((x_i-x_j)\cdot \xi) \widehat{G_i}(\xi) \overline{\widehat{G_j}(\xi)} d\sigma(\xi), 
\end{equation} 
where $G_i = \check{ g_i} \eta_R$.

We recall the standard $L^2$ estimate (see, for example, \cite[Lemma 3.2]{Guth2})
\[
\norm{\hat{f}|_S}_{L^2(d\sigma)} \les R^{1/2} \norm {f}_{L^2(B_R)}.
\]
Using this estimate and Plancherel's theorem, we bound the first term in \eqref{eqn:L2} by 
\[ R \sum_i \norm{G_i}_{L^2(\R^d)}^2 \les  R \sum_i \norm{g_i}_{L^2(\R^d)}^2.\]

The integral in the second term in \eqref{eqn:L2} is $\tilde{G_i}* \overline{G_j} *\widehat{d\sigma} (x_i-x_j)$, where $\tilde{G_i}(x) = {G_i(-x)}$. We use the decay of $\widehat{d\sigma}$ and the sparsity assumption together with the fact that $\tilde{G_i} * \overline{G_j}$ is supported in $B_{4R}(0)$ to obtain
\begin{align*}
|\tilde{G_i}* \overline{G_j} *\widehat{d\sigma} (x_i-x_j)| &\les |x_i-x_j|^{-\rho} \norm{G_i}_{L^1(\R^d)}\norm{G_j}_{L^1(\R^d)} \\
&\les  (RN)^{-C(d,\rho) \rho} R^d \norm{g_i}_{L^2(\R^d)}\norm{g_j}_{L^2(\R^d)}
\end{align*}
by the Cauchy-Schwarz inequality and Plancherel's theorem. Recall that $C(d,\rho) \rho= (d-1)$. Summing this over $i,j$ using Cauchy-Schwarz, we bound \eqref{eqn:L2} by
\[ R \sum_i \norm{g_i}_{L^2(\R^d)}^2 + (RN)^{-(d-1)} R^d N \sum_i \norm{g_i}_{L^2(\R^d)}^2 \les R \sum_i \norm{g_i}_{L^2(\R^d)}^2,\]
which completes the proof of \eqref{eq:Tao} for the case $q=2$. We remark that the proof, in fact, required a weaker sparseness condition $|x_i-x_j| \geq (RN^{1/(d-1)})^{C(d,\rho)}$ for $i\neq j$. 
\end{proof}


\end{document}